\documentclass[11pt]{amsart}

\usepackage{amssymb,bbm}

\newtheorem{theorem}{Theorem}[section]

\newtheorem{prop}[theorem]{Proposition}
\newtheorem{claim}[theorem]{Claim}

\newtheorem{lemma}[theorem]{Lemma}
\newtheorem{cor}[theorem]{Corollary}

\newtheorem{question}{Question}

\theoremstyle{definition}
\newtheorem{definition}[theorem]{Definition}

\theoremstyle{remark}

\newtheorem{fact}[theorem]{Fact}

\newcommand{\om}{\omega}
\newcommand{\D}{\mathcal{D}}
\newcommand{\rst}{\upharpoonright}
\newcommand{\crit}{\operatorname{crit}}
\newcommand{\1}{\mathbbm{1}}

\newcommand{\dom}{\mathop{\mathrm{dom}}\nolimits}

\newcommand{\la}{\langle}
\newcommand{\ra}{\rangle}

\newcommand{\Ord}{\operatorname{Ord}}

\newcommand{\lev}[2]{\sigma_{#1}(#2)}
\newcommand{\flev}[3]{\dot{\sigma}^{#1}_{#2}(#3)}

\newcommand{\be}{\begin{eqnarray}}
\newcommand{\ee}{\end{eqnarray}}

\newcommand{\Lev}[2]{\operatorname{Lev}_{#1}(#2)}
\newcommand{\ITP}{\operatorname{ITP}}
\newcommand{\ISP}{\operatorname{ISP}}
\newcommand{\TP}{\operatorname{TP}}
\newcommand{\Col}{\operatorname{Col}}
\newcommand{\defn}{\emph}

\title{The super tree property at the successor of a singular}
\author{Sherwood Hachtman}

\address{Department of Mathematics, Statistics, and Computer Science\\
University of Illinois at Chicago\\
Chicago, IL 60613, USA}

\email{hachtma1@uic.edu}

\thanks{Hachtman was partially supported by the National Science Foundation, grant DMS-1246844.}

\author{Dima Sinapova}

\address{Department of Mathematics, Statistics, and Computer Science\\
University of Illinois at Chicago\\
Chicago, IL 60613, USA}

\email{sinapova@uic.edu}

\thanks{Sinapova was partially supported by the National Science Foundation, grants DMS-1362485 and Career-1454945.}

\begin{document}

\begin{abstract}
   For an inaccessible cardinal $\kappa$, the super tree property (ITP) at $\kappa$ holds if and only if $\kappa$ is supercomact. However, just like the tree property, it can hold at successor cardinals. We show that ITP holds at the successor of the limit of $\omega$ many supercompact cardinals. Then we show that it can consistently hold at $\aleph_{\omega+1}$. We  also consider a stronger principle, $\ISP$ and certain weaker variations of it.  We determine which level of ISP can hold at a successor of a singular. These results fit in the broad program of testing how much compactness can exist in the universe, and obtaining large cardinal-type properties at smaller cardinals.
\end{abstract}
\bibliographystyle{plain}

\maketitle

\section{Introduction}

A major theme in set theory is the question of how much compactness can consistently exist in the universe.   We regard an object as satisfying a {\it compactness principle} if whenever a property holds for all strictly smaller substructures of the object, then this property holds for the object itself.  Such principles typically follow from large cardinals, but can often occur at successor cardinals as well.  So the aforementioned question can be rephrased as:  What combinatorial properties of large cardinals can consistently hold at small ones?

A key instance of such a combinatorial principle is the \emph{tree property}.  A regular cardinal $\mu$ has the tree property if every tree of height $\mu$, all of whose levels have size less than $\mu$, has a cofinal branch.  For inaccessible $\mu$, the tree property is equivalent to weak compactness of $\mu$; but this combinatorial principle can consistently hold at successor cardinals.  Results of Mitchell and Silver show that the tree property at $\aleph_2$ is equiconsistent with a weakly compact cardinal.  Mitchell's 1972 proof \cite{mitchell} obtaining the consistency of the tree property at $\aleph_2$ initiated a long, ongoing project in set theory: Obtain the tree property at all regular cardinals greater than $\aleph_1$.

There are strengthenings of the tree property that capture the essence of larger cardinals in a similar way.  Jech defined a principle which we call the strong tree property that characterizes strongly compact cardinals.  Then Magidor isolated a further strengthening, $\ITP$ (or the super tree property) that can characterize supercompact cardinals. We will give the precise definitions in the next section, but the highlight is the following:
\begin{fact} Let $\mu$ be an inaccessible cardinal.
\begin{enumerate}
\item (Jech, 1973 \cite{Jech}) The strong tree property holds at $\mu$ if and only if $\mu$ is strongly compact.
\item
(Magidor, 1974 \cite{Magidor}) $\ITP$ holds at $\mu$ if and only if $\mu$ is supercompact.
\end{enumerate}
\end{fact}

Much as with the tree property, these combinatorial characterizations can consistently hold at successor cardinals.  Indeed, if one starts with a supercompact $\lambda$ and forces with the Mitchell poset to make the tree property hold at $\lambda=\aleph_2=2^\omega$, then even $\ITP$ will hold at $\aleph_2$ in the generic extension.   Similar remarks hold for the strong tree property.

So we have an even more ambitious version of the project to obtain the tree property everywhere: Can we consistently obtain the strong or super tree properties at every regular cardinal greater than $\aleph_1$?

Knowing that the tree property and its strenghtenings can be obtained at $\aleph_2$, the natural follow up is what happens at higher $\aleph_n$'s. Below we summarize some history on the subject:
\begin{enumerate}
\item
(Abraham, 1983 \cite{abraham}) Starting from a supercompact and a weakly compact above it, one can force the tree property simultaneously at $\aleph_2$ and $\aleph_3$.
\item
(Cummings-Foreman, 1998 \cite{Cummings-Foreman}) Starting from $\omega$-many supercompact cardinals, the tree property can be forced to hold simultaneously at every $\aleph_n$, for $n>1$.
\item
(Fontanella \cite{FontanellaCF}; Unger \cite{UngerCF} independently, 2013-4), In the Cummings-Foreman model, $\ITP$ holds an every $\aleph_n$, for $n>1$.
\end{enumerate}
A classical theorem of Aronszajn shows the tree property, hence also its strengthenings, fail at $\aleph_1$.  An old generalization by Specker of this theorem implies that even obtaining the tree property at $\nu^+$ and $\nu^{++}$ with $\nu$ strong limit requires a violation of the singular cardinals hypothesis (SCH) at $\nu$.  We note that obtaining this situation just at $\nu = \aleph_{\omega}$ is an open problem.  The crux of this program will thus likely be encountered at successors of singulars.

Our focus here is on this region.  The first result in that direction is by
Magidor and Shelah \cite{MagidorShelah}, who showed that if $\nu$ is a singular limit of supercompact cardinals, then $\mu = \nu^+$ has the tree property. They also showed that the tree property can be forced at $\aleph_{\omega+1}$. The original large cardinal hypothesis included a huge cardinal. It was later reduced to $\omega$-many supercompact cardinals, using a Prikry construction in \cite{sinapova-TP}. More recently,
Neeman \cite{Neeman} showed that from this situation, one can force $\mu = \aleph_{\omega+1}$ to have the tree property with a product of Levy collapses.

A few years ago, Fontanella \cite{Fontanella} generalized both arguments in \cite{MagidorShelah} and \cite{Neeman}, showing that the strong tree property holds the successor of a singular limit of strongly compact cardinals, and also can consistently hold at $\aleph_{\omega+1}$.

In this paper we show that even $\ITP$ can hold at the successor of a singular cardinal. We prove the following theorems:

\begin{theorem}
Suppose that $\langle\kappa_n\rangle_{n<\omega}$ is an increasing sequence of supercompact cardinals, $\nu = \sup_n \kappa_n$, and $\mu=\nu^+$. Then $\ITP$ holds at $\mu$.
\end{theorem}
Then we show this can be forced at smaller cardinals:
\begin{theorem}
Suppose that $\langle\kappa_n\rangle_{n<\omega}$ is an increasing sequence of supercompact cardinals, $\nu = \sup_n \kappa_n$ and $\mu=\nu^+$. Then there is a forcing extension in which $\mu=\aleph_{\omega+1}$ and $\ITP$ holds at $\aleph_{\omega+1}$.
\end{theorem}

We also consider a strengthening of $\ITP$, the so-called ineffable slender list property ($\ISP$).  This principle has been of interest in connection with the SCH as well as that of the consistency strength of the proper forcing axiom (PFA). Viale and Wei{\ss} showed that under PFA, $\ISP$ holds at $\aleph_2$ \cite{viale-weiss}.  In \cite{viale} Viale showed that $\ISP$ at $\aleph_2$ together with stationary many internally unbounded models implies that SCH holds; it is still open whether $\ISP$ by itself is enough.  Viale and Weiss gave a striking application \cite{viale-weiss}, showing that any standard iteration to force PFA must start with a strongly compact cardinal.  If in addition the iteration is proper, then there must be a supercompact cardinal in the ground model.

We consider both $\ISP$ and certain weakenings, $\ISP(\delta, \mu, \lambda)$.  Here $\ISP_\mu$ as defined by Weiss corresponds to $\ISP(\aleph_1, \mu, \lambda)$ for all $\lambda$; the principle is weakened as $\delta$ is increased; and $\ISP(\mu, \mu, \lambda)$ for all $\lambda$ implies $\ITP_\mu$. The precise definitions are in the next section.  We determine exactly which level of ISP can hold at a successor of a singular.
\begin{theorem}
Suppose $\nu$ is a singular strong limit cardinal, and $\mu = \nu^+$.  Then for all $\delta \leq \nu$, $\ISP(\delta,\mu,2^\nu)$ fails.  In particular, if SCH holds at $\nu$, then $\ISP(\nu,\mu,\mu)$ fails.
\end{theorem}
\begin{theorem}\label{ISP}
Suppose that $\langle\kappa_n\mid n<\omega\rangle$ is an increasing sequence of supercompact cardinals with limit $\nu$, $\mu=\nu^+$, and $\lambda \geq \mu$. Then $\ISP(\mu,\mu,\lambda)$ holds.
\end{theorem}

This paper is organized as follows. In section 2, we give the definitions of the principles discussed above and fix some notation. In section 3 we prove that $\ITP$ holds at successor of limit of supercompacts. In section 4, we prove the theorems regarding $\ISP$, and in section 5 we prove the consistency of $\ITP$ at $\aleph_{\omega+1}$. We then conclude with some open questions.

\section{Preliminaries and notation}

In this section we define the notion of lists and the strengthenings of the tree property discussed in the last section.
\begin{definition} Suppose $\mu$ is a regular cardinal and $\lambda \geq \mu$.  We say that ${d} = \la d_z \ra_{z \in \mathcal{P}_\mu(\lambda)}$ is a \defn{$\mathcal{P}_\mu(\lambda)$-list} if for all $z \in \mathcal{P}_\mu(\lambda)$, $d_z\subseteq z$.  A tree-like structure is obtained from a list by regarding the levels, indexed by $z \in \mathcal{P}_\mu(\lambda)$, as consisting of restrictions of $d_y$'s above:
  \[
   \Lev{d}{z} = \{d_y \cap z \mid z \subseteq y, y \in \mathcal{P}_\mu(\lambda)\}.
  \]
A \defn{cofinal branch through $d$} is a set $b\subseteq \lambda$ so that for all $z \in \mathcal{P}_\mu(\lambda)$, $b \cap z \in \Lev{d}{z}$.

A $\mathcal{P}_\mu(\lambda)$-list is \defn{thin} if $|\Lev{d}{z}|<\mu$ for all $z \in \mathcal{P}_\mu(\lambda)$.

\end{definition}
Note that if $\mu$ is inaccessible, then every $\mathcal{P}_\mu(\lambda)$-list is thin.

\begin{definition}
We say \defn{$\TP(\mu,\lambda)$ holds} if for every thin $\mathcal{P}_\mu(\lambda)$-list $d$, there is a cofinal branch $b$ through $d$.  The \defn{strong tree property} holds at $\mu$ if $\TP(\mu,\lambda)$ holds for all $\lambda \geq \mu$.

A set $b\subseteq \lambda$ is an \defn{ineffable branch} through $d$ if $\{z \in \mathcal{P}_\mu(\lambda) \mid b \cap z = d_z\}$ is stationary.  We say \defn{$\ITP(\mu,\lambda)$ holds} if every thin $\mathcal{P}_\mu(\lambda)$-list has an ineffable branch.  The \defn{super tree property} holds at $\mu$ ($\ITP_\mu$ holds) if $\ITP(\mu,\lambda)$ holds for all $\lambda \geq \mu$.
\end{definition}
Thus the levels of a list $d$ consist of small approximations to a subset of $\lambda$, and a cofinal branch is a subset of $\lambda$ that is approximated at every level.  Note that $\TP(\mu,\mu)$ is equivalent to the tree property at $\mu$.

A further strengthening of $\ITP$ called $\ISP$ was defined by Wei{\ss} in \cite{Weiss}.

\begin{definition}
A $\mathcal{P}_\mu(\lambda)$-list is \defn{slender} if for all sufficiently large $\theta$, for club many $M\in\mathcal{P}_\mu(H_\theta)$, for all $b\in M\cap\mathcal{P}_{\omega_1}(\lambda)$,  $d_{M\cap\lambda}\cap b\in M$.
$\ISP(\mu)$ holds if for every $\lambda\geq \mu$, every  slender $\mathcal{P}_\mu(\lambda)$ list has an ineffable branch.
\end{definition}

All thin lists are slender, though the converse can fail.  Consequently, $\ISP$ implies $\ITP$.  Moreover, like $\ITP$, $\ISP$ can consistently hold at $\aleph_2$.  For example, in \cite{viale-weiss} it is shown that PFA implies $\ISP$ holds at $\aleph_2$.

Viale and Wei{\ss} gave a characterization of $\ISP$ via \defn{guessing models}. We will discuss this in more detail later, together with a refinement of slenderness to analyze what happens when $\mu$ is the successor of a singular cardinal.

We use standard notation. For conditions in a forcing poset $\mathbb{P}$, $p\leq q$ denotes that $p$ is stronger than $q$. We say that $\mathbb{P}$ is $\kappa$-closed to mean that decreasing sequences of length less than $\kappa$ have a lower bound. $\mathbb{P}$ is ${<}\kappa$-distributive if it adds no new sequences of size less than $\kappa$.

\section{ITP at the successor of a singular limit of supercompacts}
We begin with the simplest instance of the super tree property, namely $\ITP(\mu,\mu)$.
\begin{theorem}\label{baby}
Suppose that $\langle\kappa_n\ra_{n<\omega}$ is an increasing sequence of supercompact cardinals with limit $\nu$ and $\mu=\nu^+$. Then we have $\ITP(\mu, \mu)$.
\end{theorem}
\begin{proof}
Since $\mu$ is club in $\mathcal{P}_\mu(\mu)$, we may assume that our list is indexed by $\mu$.  So let $d=\langle d_\alpha\mid\alpha<\mu\rangle$ be a thin $\mu$-list.  %ach $d_\alpha\subset\alpha$ and for all $\alpha$, $\{d_\beta\cap\alpha\mid \beta\geq\alpha\}$ has size at most $\nu$. Call that the $\alpha$-th level of $d$ and enumerate it by $\{\sigma^\alpha_\xi\mid \xi<\nu\}$.
Each $\Lev{d}{\alpha}$ has size at most $\nu$, so we enumerate these as $\Lev{d}{\alpha} = \{ \lev{\alpha}{\xi} \mid \xi < \nu\}$.

We will show that there is a $b\subset\mu$ such that $\{\alpha<\mu\mid d_\alpha=b\cap\alpha\}$ is stationary.  The ineffability of $b$ will come from a basic fact about sets in supercompactness measures.
\begin{fact}\label{stationarysups}
  Let $\mu$ be regular and suppose $U$ is a normal measure on $\mathcal{P}_\kappa(\mu)$ with $\kappa \leq \mu$.  Then for all $A \in U$, $\{\sup x \mid x \in A\}$ is stationary.
\end{fact}

For each $n$, let $U_n$ be a normal measure on $\mathcal{P}_{\kappa_n}(\mu)$, and let $j_{U_n}$ be the corresponding ultrapower embedding.
\begin{lemma}
There are $n<\omega$, an unbounded $S\subset\mu$, and $A\in U_0$ such that for all $x\in A$ and $\alpha\in x\cap S$, there is $\xi<\kappa_n$ with $d_{\sup x}\cap\alpha=\lev{\alpha}{\xi}$.
\end{lemma}
\begin{proof}
Let $i=j_{U_0}$ and consider $id_{\sup i"\mu}$. For all $\alpha<\mu$, there is some $n$ and $\xi<i(\kappa_n)$ such that $id_{\sup i"\mu}\cap i(\alpha)=i\lev{i(\alpha)}{\xi}$.

Then for some $n$, there is a cofinal set $S\subset\mu$ such that for all $\alpha\in S$, $id_{\sup i"\mu}\cap i(\alpha)$ has index below $i(\kappa_n)$ in $\Lev{id}{i(\alpha)}$; namely,  for some $\xi< i(\kappa_n)$, $$id_{\sup i"\mu}\cap i(\alpha)=i\lev{i(\alpha)}{\xi}.$$ Then for all $\alpha\in S$, there is a measure one set $A_\alpha\in U_0$ such that for all $x\in A_\alpha$,  for some $\xi<\kappa_n$, $d_{\sup x}\cap \alpha=\lev{\alpha}{\xi}$.

Set $A:=\triangle_{\alpha\in S}A_\alpha$. Then $S, A$ are as desired.
\end{proof}

Fix $S, A, n$ as in the conclusion of the lemma.  Note that then for all $\alpha<\beta$ both in $S$, there are $\xi, \delta<\kappa_n$ such that $\lev{\alpha}{\xi}=\lev{\beta}{\delta}\cap\alpha$.

For the purposes of the following lemma, define a (not necessarily cofinal) branch through $d$ as a set $b \subseteq \mu$ so that $b \cap \alpha \in \Lev{d}{\alpha}$ for all $\alpha \leq \sup b$.  (Note $b$ may be a cofinal branch even if $b$ is bounded as a subset of $\mu$.)
\begin{lemma}
There is a sequence $\langle b_\delta\mid\delta<\kappa_n\rangle$ of (possibly bounded) branches through $d$ and a measure one set $A'\in U_0$ such that for all $x\in A'$ and $\alpha\in x$, there is $\delta<\kappa_n$ such that $d_{\sup x}\cap\alpha=b_\delta\cap\alpha$.
\end{lemma}
\begin{proof}
Let $j=j_{U_{n+1}}:V\rightarrow M$, and let $\gamma\in j(S)\setminus\sup j"\mu$.  By elementarity, for all $\alpha\in S$, there are $\xi, \delta<\kappa_n$ such that $j(\lev{\alpha}{\xi})=j\lev{j(\alpha)}{\xi}=j\lev{\gamma}{\delta}\cap j(\alpha)$.

For each $\delta<\kappa_n$, let $$b_\delta=\bigcup\{\lev{\alpha}{\xi}\mid \alpha\in S, j(\lev{\alpha}{\xi})=j\lev{\gamma}{\delta}\cap j(\alpha)\}.$$  That is, $b_\delta$ is the pullback by $j$ of ``the predecessors" of $j\lev{\gamma}{\delta}$.  We have the following:
\begin{itemize}
\item
Each $b_\delta$ is a branch, as it is the union of a coherent sequence of elements of the $\alpha$-th level of $d$ ranging over $\alpha<\mu$.
\item
There is some $\delta<\kappa_n$ such that $\{ \alpha < \mu \mid b_\delta \cap \alpha \in \Lev{d}{\alpha}\}$ is unbounded in $\mu$.  Such $b_\delta$ is a cofinal branch through $d$.
\item
For each $\alpha < \mu$, there is $\delta<\kappa_n$ such that $b_\delta\cap\alpha \in \Lev{d}{\alpha}$.

\end{itemize}
The last item follows by the second item, but we gain more information from the following direct argument.

Fix $\alpha<\mu$. Choose $x\in j(A)$, such that $\gamma\in x$ and there is some $\alpha'\in S\setminus\alpha$ with $j(\alpha')\in x$. Then we apply elementarity. In particular, for all such $x$'s there are some $\xi,\delta<\kappa_n$ such that:
\begin{itemize}
\item
$jd_{\sup x}\cap\gamma=j\lev{\gamma}{\delta}$, and
\item
$jd_{\sup x}\cap j(\alpha')=j \lev{j(\alpha')}{\xi}=j(\lev{\alpha'}{\xi})$.
\end{itemize}
Then $jd_{\sup x}\cap j(\alpha')=j\lev{\gamma}{\delta}\cap j(\alpha')=j(b_\delta)\cap j(\alpha')$.  The last equality is since by definition, $b_\delta\cap\alpha'=\lev{\alpha'}{\xi}$.
So, $jd_{\sup x}\cap j(\alpha)=j(b_\delta)\cap j(\alpha)=j(b_\delta\cap\alpha)$.

So we have in $M$ that for all $\alpha<\mu$, and for $j(U_0)$-measure one many $x\in\mathcal{P}_{\kappa_0}(j(\mu))$, there is some $\delta$, such that
$jd_{\sup x}\cap j(\alpha)=j(b_\delta\cap\alpha)$.

Note that $j(\langle b_\delta\mid \delta<\kappa_n\rangle)=\langle j(b_\delta)\mid \delta<\kappa_n\rangle$, since $\kappa_n$ is below the critical point.

Then by elementarity, in $V$ we have for all $\alpha<\mu$ that there is a measure one set $A_\alpha\in U_0$ such that for all $x\in A_\alpha$, there is $\delta<\kappa_n$ such that $d_{\sup x}\cap\alpha=b_\delta\cap \alpha$.  Set  $A':=\triangle A_\alpha$. $A'$ is as desired.
\end{proof}

Fix the branches $\langle b_\delta\mid\delta<\kappa_n\rangle$ and $A'\in U_0$ as in the above lemma. By restricting to a subset of $\kappa_n$, if necessary, assume that for $\eta<\delta$, $b_\eta$ and $b_\delta$ are distinct branches. (Note that this is not automatic for all $\eta<\delta$, since our top node $\gamma$ may be strictly above $\sup j"\mu$.)
Then for distinct $\eta, \delta<\kappa_n$, let $\alpha_{\eta, \delta}$ be such that for all $\alpha\geq\alpha_{\eta, \delta}$, $b_\eta\cap\alpha\neq b_\delta\cap\alpha$ (it exists, otherwise they will be the same branch).
Let $\bar{\alpha}=\sup_{\eta, \delta<\kappa_n}\alpha_{\eta,\delta}<\mu$.

But then for all $x\in A'$ with $\bar{\alpha}\in x$,  there is a unique $\delta<\kappa_n$, such that for all $\alpha\in x\setminus \bar{\alpha}$,
$d_{\sup x}\cap\alpha=b_\delta\cap\alpha$. By intersecting with a measure one set, we may assume that for all $x\in A'$, $x$ is unbounded in $\sup x$. Then $d_{\sup x}=b_\delta\cap\sup x$.

So the set $T:=\{\beta<\mu\mid (\exists \delta<\kappa_n) d_\beta=b_\delta\cap\beta\}\supset\{\sup x\mid \bar{\alpha}\in x, x\in A'\}$, and by Fact~\ref{stationarysups}, this last set is stationary.  For each $\delta$, let $T_\delta:=\{\beta<\mu\mid d_\beta=b_\delta\cap\beta\}$.  Since $T=\bigcup_{\delta<\kappa_n} T_\delta$, there is some $\delta$ so that $T_\delta$ is stationary.

This completes the proof of Theorem~\ref{baby}.
\end{proof}
Next, we argue for the two-cardinal version.
\begin{theorem}\label{two cardinal}
Suppose that $\langle\kappa_n\ra_{n<\omega}$ is an increasing sequence of supercompact cardinals with limit $\nu$ and $\mu=\nu^+$, and let $\lambda>\mu$ be inaccessible. Then we have $\ITP(\mu, \lambda)$.
\end{theorem}
\begin{proof}
Suppose that $d=\langle d_z\mid z\in \mathcal{P}_\mu(\lambda)\rangle$ is a thin $\mathcal{P}_\mu(\lambda)$-list. Recall for each $z\in\mathcal{P}_\mu(\lambda)$ that the $z$-th level of $d$ is $\Lev{d}{z}=\{z\cap d_y\mid y\supset z\}$.  Since $d$ is thin, we enumerate it as $\{\lev{z}{\xi}\mid \xi<\nu\}$.

For each $n$, let $U_n$ be a normal measure on $\mathcal{P}_{\kappa_n}(\lambda)$. Note that $\lambda=|\mathcal{P}_\mu(\lambda)|$.

Let $i=j_{U_0}$ and set $z^*:=\bigcup i"\mathcal{P}_\mu(\lambda)$, and let $g:\mathcal{P}_{\kappa_0}(\lambda)\rightarrow\mathcal{P}_{\mu}(\lambda)$ be such that $z^*=[g]_{U_0}$. Then $id_{z^*}$ (that is, $i(d)_{z^*}$) is $[x\mapsto d_{g(x)}]_{U_0}$.

\begin{claim}\label{g-image is stationary}
If $A\in U_0$, then $\bar{A}:=\{g(x)\mid x\in A\}$ is stationary in $\mathcal{P}_\mu(\lambda)$.
\end{claim}
\begin{proof}
Suppose that $C$ is a club in $\mathcal{P}_\mu(\lambda)$. Then in $M$, $i"C$ is a directed subset of $i(C)$ of size less than $i(\mu)$. So, $z^*=\bigcup i"\mathcal{P}_\mu(\lambda)=\bigcup i"C\in j(C)$.

Also, by definition of $g$, $z^*\in j(\bar{A})$. So, $C\cap\bar{A}$ is nonempty.
\end{proof}

For all $z\in\mathcal{P}_\mu(\lambda)$, there is some $n$ and $\xi<j(\kappa_n)$, such that $id_{z^*}\cap i(z)=i\lev{j(z)}{\xi}$
Then for some $n$, there is a stationary set $S\subset\mathcal{P}_\mu(\lambda)$, such that for all $z\in S$, there is $\xi<i(\kappa_n)$, $$id_{z^*}\cap i(z)=i\lev{i(z)}{\xi}.$$ Then for all $z\in S$, there is a measure one set $A_z\in U_0$, such that for all $x\in A_z$, there is $\xi<\kappa_n$, such that $d_{g(x)}\cap z=\lev{z}{\xi}$.

Next we want to take a diagonal intersection of the $A_z$'s. To that end, fix a bijection $c:\mathcal{P}_\mu(\lambda)\rightarrow \lambda$.
Let $h$ be a function with domain $\mathcal{P}_{\kappa_0}(\lambda)$, such that $h(x)=\{z\in \mathcal{P}_\mu(\lambda)\mid c(z)\in x\}$.
\begin{claim}
$[h]_{U_0}=j"\mathcal{P}_\mu(\lambda)$.
\end{claim}
\begin{proof}
Clearly, for each $z\in \mathcal{P}_\mu(\lambda)$, $j(z)\in [h]_{U_0}$. For the other direction,
if $[f]_{U_0}\in [h]_{U_0}$, then for $U_0$-almost every $x$, $f(x)=z\in\mathcal{P}_\mu(\lambda)$ for some $z$ with $c(z)\in x$, i.e. $c(f(x))\in x$. By normality, $c\circ f$ is constant on a measure one set, say with value $\alpha$. Setting $z:=c^{-1}(\alpha)$, we have $[f]_{U_0}=j(z)$.
\end{proof}
So let us assume that for all $x\in\mathcal{P}_{\kappa_0}(\lambda)$, $g(x)=\bigcup h(x)=\bigcup\{z\mid c(z)\in x\}$.

Now set $A:=\triangle_{z\in S}A_z=\{x\in\mathcal{P}_{\kappa_0}(\lambda)\mid x\in \bigcap_{c(z)\in x} A_z\}\in U_0$.
Then if $x\in A$, $z\in S$, and $c(z)\in x$, there is $\xi<\kappa_n$, such that $d_{g(x)}\cap z=\lev{z}{\xi}$. As a corollary, we have that for all $z\subset w$, both in $S$, there are $\xi, \delta<\kappa_n$, such that $\lev{z}{\xi}=\lev{w}{\delta}\cap z$.
\bigskip

Next we prove the analogous result from Lemma 2 in Theorem \ref{baby}.
\begin{lemma}
There is a sequence $\langle b_\delta\mid\delta<\kappa_n\rangle$ of (possibly bounded) branches through the list and a measure one set $A'\in U_0$, such that for all $x\in A'$, for all $z\in\mathcal{P}_\mu(\lambda)$ with $c(z)\in x$, there is $\delta<\kappa_n$, such that $d_{g(x)}\cap z=b_\delta\cap z$.

\end{lemma}
\begin{proof}

Let $j=j_{U_{n+1}}:V\rightarrow M$.

By elementarity, $j(A)\subset\mathcal{P}_{\kappa_0}(j(\lambda))$ is in $j(U_0)$ and $jc:\mathcal{P}_{j(\mu)}(j(\lambda))\rightarrow j(\lambda)$ is a bijection. And if $x\in j(A)$ and $z\in j(S)$ is such that $jc(z)\in x$, then there is $\delta<\kappa_n$, such that $jd_{jg(x)}\cap z=j\lev{z}{\delta}$. Now let $u\in j(S)$ be such that $\bigcup j"\mathcal{P}_\mu(\lambda)\subset u$. Then it follows that, for all $z\in S$, there are $\xi, \delta<\kappa_n$, such that $j(\lev{z}{\xi})=j\lev{j(z)}{\xi}=j\lev{u}{\delta}\cap j(z)$.

For each $\delta<\kappa_n$, let $$b_\delta=\bigcup\{\lev{z}{\xi}\mid z\in S, j(\lev{z}{\xi})=j\lev{u}{\delta}\cap j(z)\}.$$ I.e., analogously as before, $b_\delta$ is the pullback of ``the predecessors" of $j\lev{u}{\delta}$. Then
each $b_\delta$ is the union of a coherent sequence of elements of the $z$-th level of $d$ ranging over $z\in\mathcal{P}_\mu(\lambda)$ (it may be bounded). And for each $z\in\mathcal{P}_\mu(\lambda)$, there is $\delta<\kappa_n$, such that $b_\delta\cap z$ is in the $z$-th level of $d$.

\begin{claim}
For all $z\in\mathcal{P}_\mu(\lambda)$, in $M$, for $j(U_0)$-measure one many $x\in\mathcal{P}_{\kappa_0}(j(\lambda))$, there is some $\delta<\kappa_n$, such that
$jd_{jg(x)}\cap j(z)=j(b_\delta\cap z)$.
\end{claim}
\begin{proof}

Fix $z\in \mathcal{P}_\mu(\lambda)$.
Choose $x\in j(A)$, such that $jc(u)\in x$ and there is some $z'\in S$, $z\subset z'$ with $jc(j(z'))=j(c(z'))\in x$. Then there there is some $\delta<\kappa_n$, such that $jd_{jg(x)}\cap u=j\lev{u}{\delta}$, and for some $\xi<\kappa_n$, $jd_{jg(x)}\cap j(z')=j \lev{j(z')}{\xi}=j(\lev{z'}{\xi})=j\lev{u}{\delta}\cap j(z')=j(b_\delta)\cap j(z')$. The last equality is since by definition, $b_\delta\cap z'=\lev{z'}{\xi}$.

So $jd_{jg(x)}\cap j(z)=j(b_\delta)\cap j(z)=j(b_\delta\cap z)$.
\end{proof}

Then by  elementarity, in $V$, for all $z\in\mathcal{P}_\mu(\lambda)$, there is a measure one set $A_z\in U_0$, such that for all $x\in A_z$, there is $\delta<\kappa_n$, such that $d_{g(x)}\cap z=b_\delta\cap z$. Set  $A':=\triangle A_z=\{x\mid x\in\bigcap_{c(z)\in x}A_z\}$. This is as desired.
\end{proof}

Fix the branches $\langle b_\delta\mid\delta<\kappa_n\rangle$ and $A'\in U_0$ as in the above lemma. By passing to a subset of $\kappa_n$ if necessary, assume that for $\eta<\delta$, $b_\eta$ and $b_\delta$ are distinct branches. As before, for $\eta<\delta<\kappa_n$, let $z_{\eta, \delta}$ be such that for all $z\supset z_{\eta, \delta}$, $b_\eta\cap z\neq b_\delta\cap z$.
Let $\bar{z}=\bigcup_{\eta<\delta<\kappa_n}z_{\eta,\delta}\in\mathcal{P}_\mu(\lambda)$.

Let $A''=\{x\in A'\mid c(\bar{z})\in x, g(x)=\bigcup\{z\mid \bar{z}\subset z, c(z)\in x\}\}\in U_0$. Then for all $x\in A''$, there is a unique $\delta<\kappa_n$, such that for all $z\in \mathcal{P}_\mu(\lambda)$, with $c(z)\in x, \bar{z}\subset z$, we have
$d_{g(x)}\cap z=b_\delta\cap z$. It follows that $d_{g(x)}=b_\delta\cap g(x)$.

So the set $T:=\{z\in\mathcal{P}_\mu(\lambda)\mid (\exists \delta<\kappa_n) d_z=b_\delta\cap z\}\supset\{g(x)\mid  x\in A''\}$, which is stationary by Claim \ref{g-image is stationary}. For each $\delta$, let $T_\delta:=\{z\in\mathcal{P}_\mu(\lambda)\mid d_z=b_\delta\cap z\}$. Since $T=\bigcup_{\delta<\kappa_n} T_\delta$, for some $\delta$, $T_\delta$ is stationary.

This completes the proof of Theorem \ref{two cardinal}.
\end{proof}
\section{ISP at the successor of a singular cardinal}
In this section we analyze a somewhat stronger principle at successor of a singular, called ISP.
Let us begin with some definitions.

\begin{definition}
  Let $M \prec H_\theta$ for some $\theta$, and suppose $z \subseteq a$ for some $a \in M$, and $\delta \in M$ is a cardinal.  We say \emph{$M$ $\delta$-approximates $z$} if for all $x \in M \cap \mathcal{P}_\delta(a)$, we have $x \cap z \in M$.

  Let $\delta \leq \mu \leq \lambda$ be cardinals with $\mu$ regular.  A $\mathcal{P}_\mu(\lambda)$-list $\la d_x \ra_{x \in \mathcal{P}_\mu(\lambda)}$ is \emph{$\delta$-slender} if for all cardinals $\theta$ that are sufficiently large, the set
    \[
     \{ M \in H_\theta \mid \text{$M$ $\delta$-approximates }d_{M \cap \lambda}\}
    \]
contains a club.

  We say the principle $\ISP(\delta,\mu,\lambda)$ holds if every $\delta$-slender $\mathcal{P}_\mu(\lambda)$-list has an ineffable branch.
\end{definition}

  This principle is made stronger if $\delta$ is decreased or $\lambda$ is increased.  Note that $\ISP(\mu,\lambda)$ as originally defined by Wei{\ss} in \cite{Weiss} is equivalent to our $\ISP(\aleph_1,\mu,\lambda)$.  Moreover, $\ISP(\mu,\mu,\lambda)$ for all $\lambda$ implies $\ITP_\mu$.

  In \cite{viale-weiss}, Viale and Wei{\ss} gave a characterization of $\ISP$ via  guessing models.  A model $M$ is \defn{$\delta$-guessing} if whenever $M$ $\delta$-approximates $x$ with $x$ a subset of some $a \in M$, then $x$ is $M$-guessed, i.e. there is $b\in M$ such that $b\cap a=x$.  For more on these objects, see \cite{VialeGuessing}.

  Viale and Weiss showed that if $\ISP(\aleph_1,\mu,|H_\theta|)$ holds, then there are stationarily many $\aleph_1$-guessing models $M \prec H_\theta$ with $|M|<\mu$; and $\ISP(\aleph_1,\mu,\lambda)$ holds for all $\lambda \geq \mu$ if and only if there are stationarily many $\aleph_1$-guessing models of size less than $\mu$ in $H_\theta$ for all large $\theta$.  Similarly, we have:

  \begin{fact}\label{ISPgivesguessers}
  $\ISP(\delta,\mu,|H_\theta|)$ implies that there are stationary many $\delta$-guessing models $M\prec H_\theta$ with $|M|<\mu$.
  \end{fact}

  Let us now observe a limitation on the extent to which this principle can hold for $\mu$ the successor of a strong limit singular cardinal. In \cite{Sherwood-Dima}, it was shown that the principle $\ISP$ as defined by Wei{\ss} (i.e. $\ISP(\aleph_1,\mu,\lambda)$ for all $\lambda \geq \mu$) cannot hold at the single or at the double successor of a singular strong limit cardinal.  The next theorem generalizes this fact.  Note the proof of Fact~\ref{ISPgivesguessers} is embedded in this argument.

  \begin{theorem} Suppose $\nu$ is a singular strong limit cardinal, and $\mu = \nu^+$.  Then $\ISP(\nu,\mu,2^\nu)$ fails.
  \end{theorem}

  \begin{proof}
Suppose for contradiction $\ISP(\nu,\mu,2^\nu)$ holds; write $\tau = 2^\nu$.  We seek a $\nu$-slender $\mathcal{P}_{\mu}(\tau)$ list $d$ with no ineffable branch.

Note that $2^\nu=|H_\mu|$; fix a bijection $f:\tau \to H_\mu$.  There is a club $C \subseteq \mathcal{P}_{\mu}(H_\mu)$ of $M$ so that $M \prec H_\mu$, $f \in M$, and $M \cap \mu$ is an ordinal $\alpha$ with $\nu < \alpha < \mu$.  Note that no such $M$ can be $\nu$-guessing: Since $\nu$ is strong limit, $2^{<\nu} \subseteq M$, and so every subset of $\nu$ is trivially $\nu$-approximated by $M$; but if $M$ guessed every subset of $\nu$, we'd have $\mathcal{P}(\nu) \subseteq M$, contradicting $|M|<\mu\leq 2^\nu$.

So for each $z \in \mathcal{P}_\mu(\tau)$ such that $f"z \in C$, denote $M_z=f"z$ and let $x_z$ be a subset of $\nu$ that is not $M_z$-guessed.  Then put $d_z = \{\alpha \in \tau \mid f(\alpha) \in x_z\}$.  Since $f \in M_z$, $d_z \subseteq z$, and
(any $\mathcal{P}_\mu(\tau)$-list extending) $\la d_z \ra_{M_z \in C}$ is clearly $\nu$-slender.

Let $b$ be an ineffable branch through $d$.  Since by construction $f"d_z \subseteq \nu$ for a club of $z$, we have $f"b \subseteq \nu$.  So fix $z$ such that $M_z \in C$, $f"b \in M_z$, and $d_z = b \cap z = b$.  But $x_z = f"d_z = f"b \in M_z$ was defined as a subset of $\nu$ that was not $M_z$-guessed.  This is a contradiction.
  \end{proof}

It turns out that the above theorem is sharp. In particular, next we show that a modification of the arguments from the previous section yield $\ISP(\mu,\mu,\lambda)$ when $\mu$ is the successor of a limit of supercompacts (note that in this situation, $2^\nu = \mu$).

\begin{theorem}\label{ISP}
Suppose that $\langle\kappa_n\mid n<\omega\rangle$ is an increasing sequence of supercompact cardinals with limit $\nu$, $\mu=\nu^+$, and $\lambda \geq \mu$. Then $\ISP(\mu,\mu,\lambda)$ holds.
\end{theorem}

\begin{proof}
  We are given $d = \la d_x \mid x \in \mathcal{P}_\mu \lambda \ra$ a $\mu$-slender list.  Let $\theta \gg \lambda$, and fix a $\theta$-supercompactness measure $U_0$ on $\mathcal{P}_{\kappa_0}(\theta)$, with embedding $i=j_{U_0}:V\to M$.  In particular, $i(\kappa_0) > \theta$ and $i"\theta \in M$.

  Let $N^* = \bigcup i" \mathcal{P}_\mu H_\theta$.  Then $N^*\in M$ and in $M$, $|N^*| = i(\nu)$ and $N^* \subseteq H_{i(\theta)}$.  Let $g:\mathcal{P}_{\kappa_0}(\theta) \to \mathcal{P}_\mu H_\theta$ be such that $[g]_{U_0} = N^*$.

  As before, we have that for all clubs $C \subseteq \mathcal{P}_\mu H_\theta$ in $V$, $N^* \in i(C)$.  In particular, by slenderness of $d$, $M$ satisfies that $N^*$ $i(\mu)$-approximates $id_{N^* \cap i(\lambda)}$.  Thus for any $z \in \mathcal{P}_\mu(\lambda)$, we have $i(z) \cap id_{N^*\cap i(\lambda)} \in N^*$.

  It follows that for each $z\in \mathcal{P}_\mu(\lambda)$, there is an elementary substructure $M_z \in \mathcal{P}_\mu(H_\theta)$ such that $z \cup \{z\} \subseteq M_z$, and
  \[i(z) \cap id_{N^*\cap i(\lambda)} \in i(M_z);
  \]
 we may assume $|M_z|=\nu$.  For each such $z$, let us enumerate $\mathcal{P}_\mu(z) \cap M_z$ as $\la \lev{z}{\xi} \mid \xi < \nu \ra$.

  Let $c:\mathcal{P}_\mu(\lambda) \to \theta$ be injective.
  \begin{lemma} There exist $n< \omega$, a stationary $S\subseteq \mathcal{P}_\mu(\lambda)$, and $A \in U_0$ such that for all $z \in S$ and $x \in A$ with $c(z)\in x$, there is some $\xi < \kappa_n$ such that $z \cap d_{g(x) \cap \lambda} = \lev{z}{\xi}$.
  \end{lemma}
  \begin{proof}
    For each $z \in \mathcal{P}_\mu(\lambda)$, there is some $\xi < i(\nu)$ so that $i(z) \cap id_{N^* \cap i(\lambda)} = i\lev{i(z)}{\xi}$; let $n_z$ be least so that $\xi < i(\kappa_n)$.

    The map $z \mapsto n_z$ is constant on a cofinal $S$, say with value $n$.  For each $z \in S$, we have by {\L}o\'{s}'s theorem
      \[
      A_z = \{x \in \mathcal{P}_{\kappa_0} (\theta) \mid \exists \xi < \kappa_n  z \cap d_{g(x) \cap \lambda} = \sigma_{z}(\xi) \} \in U_0.
      \]
As before we wish to take a diagonal intersection of the $A_z$ over $z \in S$.  Recall that we fixed an injective $c:\mathcal{P}_\mu(\lambda) \to \theta$; define $h:\mathcal{P}_{\kappa_0}(\theta) \to \mathcal{P}(\mathcal{P}_\mu(\lambda))$ by
  \[
    h(x) = \{z \in \mathcal{P}_\mu(\lambda) \mid c(z) \in x\}.
  \]
It's not hard to see that $[h] = i"\mathcal{P}_\mu(\lambda)$; and for $U_0$-many $x$, $g(x) \cap \lambda = \bigcup h(x)$.  Then let
  \[
  \triangle_{z \in S} A_z = \{a \in \mathcal{P}_{\kappa_0}(\theta) \mid (\forall z \in S) c(z) \in x \to x \in A_z\}.
  \]
Then $A = \triangle_{z \in S} A_z$ is as in the statement of the lemma.
\end{proof}

Next, we reduce the number of potential branches as before.  Let $j=j_{U_{n+1}}:V\to M'$ witness $\theta$-supercompactness of $\kappa_{n+1}$.  By elementarity, $M'$ satisfies that $j(S)$ is stationary in $\mathcal{P}_{j(\mu)}(j(\lambda))$, and that for all $z \in j(S)$, if $jc(z) \in x \in j(A)$, then $jd_{jg(x)\cap j(\lambda)} \cap j(z) = j\lev{j(z)}{\xi} = j(\lev{z}{\xi})$ for some $\xi < j(\kappa_n)=\kappa_n$.

Let $u \in j(S)$ satisfy $u \supseteq \bigcup j" \mathcal{P}_\mu(\lambda)$.  Then for any $z \in S$, we have $j(z),u \in j(S)$, and there is some $x \in j(A) \subseteq \mathcal{P}_{\kappa_0}(j(\theta))$ so that $jc(j(z))$ and $jc(u)$ are both in $x$.  Since $j(z) \subseteq u$, it follows that for some $\xi,\delta < \kappa_n$,
\[
j(\lev{z}{\xi}) = j\lev{j(z)}{\xi} = jd_{jg(x) \cap j(\lambda)} \cap j(z) = jd_{jg(x)\cap j(\lambda)} \cap j(z) = j\lev{u}{\delta} \cap j(z).
\]
Define, for $\delta < \kappa_n$,
 \[
  b_\delta = \bigcup\{\lev{z}{\xi} \mid j(\lev{z}{\xi}) = j\lev{u}{\delta} \cap j(z)\}.
 \]
We have just shown that for every $z \in S$, that there are some $\xi,\delta<\kappa_n$ with $b_\delta \cap z = \lev{z}{\xi}$. And more precisely, for every $z\in S$ and $U_0$-a.e.\ $x \in \mathcal{P}_{\kappa_0}(\theta)$, there is $\xi<\kappa_n$ such that $d_{g(x)\cap \lambda} \cap z = \lev{z}{\xi}$.

Let $A'_z$ be this measure one set.  We again take the diagonal intersection, $A' = \triangle_{z \in S} A'_z$:
 \[
 A' = \{x \in \mathcal{P}_{\kappa_0}(\theta) \mid (\forall z \in S) c(z) \in x \to \exists \delta<\kappa_n, d_{g(x)\cap\lambda} \cap z = b_\delta \cap z\}.
 \]
Again, by passing to a subset of $\kappa_n$ if necessary, we assume that for all $\delta<\eta$, $b_\delta\neq b_\eta$. Take $\bar{z}$ ``above all the splitting'', so that $b_\eta \cap \bar{z} \neq b_\delta \cap \bar{z}$ for all $\eta \neq \delta < \kappa_n$.  Let
\[
 T = \{z \in \mathcal{P}_\mu(\lambda) \mid \bar{z} \subseteq z \text{ and for some }\delta<\kappa_n, \; d_z = b_\delta \cap z\}.
\]
Now by the above remarks, there are measure one many $x \in A'$ so that $d_{g(x) \cap \lambda} = d_{\bigcup \{z \mid c(z) \in x\}}$.  Thus $T \supset\{g(x)\cap \lambda \mid x \in A'\}$, and so by (the same argument as in) Claim~\ref{g-image is stationary}, $T$ is a stationary set.  Again, there is some $\delta$ so that $T_\delta = \{z \in \mathcal{P}_\mu(\lambda) \mid d_z = b_\delta \cap z\}$ is stationary.  This $b_\delta$ is the desired ineffable branch.
\end{proof}

Combining Fact~\ref{ISPgivesguessers} with the previous two theorems, we obtain the following Corollary, answering Questions 8.2 and 8.3 of Viale in \cite{VialeGuessing}.

\begin{cor} Suppose $\mu =\nu^+$ with $\nu$ the limit of $\omega$-many supercompacts.  Then for all cardinals $\theta$ taken sufficiently large, there are stationarily many $\mu$-guessing models of size $\nu$ in $H_\theta$; and none of these is $\delta$-guessing for any $\delta \leq \nu$.
\end{cor}

\section{ITP at $\aleph_{\omega+1}$}
We next show, assuming the existence of infinitely many supercompacts, that it is consistent for $\aleph_{\omega+1}$ to have the super tree property.  We begin by showing it is consistent to have $\ITP(\aleph_{\omega+1},\aleph_{\omega+1})$.  The forcing will almost be the same as that used by Neeman \cite{Neeman} to obtain the tree property.  We first take the product of Levy collapses to turn $\kappa_{n}$ into $\kappa_0^{+n}$; we then show there exists some inaccessible $\rho<\kappa_0$ so that collapsing to make $\rho^{+\omega+1}$ become $\aleph_1$, and $\kappa_0$ become $\aleph_3$, forces the tree property at $\aleph_{\omega+1}$.  In fact the argument will show there are measure one many (in the normal measure on $\kappa_0$ induced by our supercompactness measure) such $\rho$.

Suppose that $\langle\kappa_n\mid n<\omega\rangle$ is an increasing sequence of indestructibly supercompact cardinals, $\nu:=\sup_n\kappa_n$, $\mu=\nu^+$.

For a successor cardinal $\tau=\delta^+$, where $\delta<\kappa_0$ is a singular cardinal of countable cofinality, let $\mathbb{L}_\tau:=\Col(\omega, \delta)\times \Col(\tau^+, {<}\kappa_0)$.  Note that forcing with $\mathbb{L}_\tau$ makes $\tau$ into $\aleph_1$ and $\kappa_0$ into $\aleph_3$ in the extension.

\begin{theorem}\label{forcebaby}
Let $H=\prod_{n}H_n$ be $\prod_{n<\om} \Col(\kappa_n, {<}\kappa_{n+1})$-generic over $V$.  Then there is a $\tau<\kappa$ such that $\tau = \delta^+$ for $\delta$ a strong limit cardinal of cofinality $\omega$, and in the extension of $V[H]$ by $\mathbb{L}_\tau$, $\ITP(\mu, \mu)$ holds.
\end{theorem}
Work in $V[H]$.  Supposing otherwise, we have, for every such $\tau<\kappa_0$, a $\mathbb{L}_\tau$ name for a thin $\mu$-list $\dot{d}^\tau$ forced by $\mathbb{L}_\tau$ to have no ineffable branch.  Assume that for all $\alpha<\mu$, $\1_{\mathbb{L}_\tau}$ forces that the $\alpha$-th level of $\dot{d}^\tau$ is enumerated by the names $\{\flev{\tau}{\alpha}{\xi}\mid \xi<\nu\}$; we may furthermore assume that for sufficiently large $\alpha < \mu$, it is forced that there are no repetitions in the sequence $\la \sigma^\tau_\alpha(\xi)\mid\xi<\nu\ra$.

By indestructibility, let $U_0$ be a normal measure on $\mathcal{P}_{\kappa_0}(\mu)$ in $V[H]$, and for each $n>0$, let $U_n$ be a normal measure on $\mathcal{P}_{\kappa_n}(\mu)$ in $V$.

Let $i=j_{U_0}:V \to M_0$ be the ultrapower embedding; for ease of notation we set $\kappa := \kappa_0$.  Recall for $x \in \mathcal{P}_{\kappa}(\mu)$ that $\kappa_x = \sup (x \cap \kappa)$, (which is just $x\cap\kappa$ on a measure one set), and $[x \mapsto \kappa_x]_{U_0} = \kappa$.  Therefore $x \mapsto \kappa_x^{+\omega+1}$ represents $\mu = \kappa^{+\omega+1}$ in $M_0$.  Let $\mu_x$ denote $\kappa_x^{+\omega+1}$ in what follows.
\begin{lemma}\label{forcelemma1}
There exist $n<\omega$, an unbounded $S\subset\mu$, and $A\in U_0$ and a map $x\mapsto (p_x,q_x)$, such that for all $x\in A$ and $\alpha\in x\cap S$, there is $\xi<\kappa_n$, such that $(p_x,q_x)\Vdash_{\mathbb{L}_{\mu_x}} \dot{d}^{\mu_x}_{\sup x}\cap \alpha=\flev{\mu_x}{\alpha}{\xi}$.
\end{lemma}
\begin{proof}
By the above remarks, $[x\mapsto\mathbb{L}_{\mu_x}]_{U_0}=\Col(\omega, \nu)\times \Col(\mu^+, {<}i(\kappa))=\mathbb{L}_\mu$.  Now $i\dot{d}^\mu_{\sup i"\mu}=[x\mapsto\dot{d}^{\mu_x}_{\sup x}]_{U_0}$. For all $\alpha<\mu$, there is some $n_\alpha$, $\xi<i(\kappa_n)$, and $(p_\alpha,q_\alpha)\in\mathbb{L}_\mu$ such that $(p_\alpha,q_\alpha)\Vdash i\dot{d}^\mu_{\sup i"\mu}\cap i(\alpha)=i\flev{\mu}{i(\alpha)}{\xi}$.  Choosing the $q_\alpha$ inductively, we arrange that $\langle q_\alpha\mid \alpha<\mu\rangle$ is decreasing.

Note that $p_\alpha \in \Col(\omega,\nu)$ are finite conditions; then there are an unbounded $S\subset \mu$ and fixed $n$ and $p$ such that for all $\alpha\in S$, $p_\alpha=p$ and $n_\alpha=n$.  By $\mu$-closure of $\Col(\mu^+,{<}i(\kappa))$, we can take $q$ to be a common strengthening of the $q_\alpha$. Let $[x\mapsto p_x]_{U_0}=p$ and $[x\mapsto q_x]_{U_0}=q$.

Then for all $\alpha\in S$, there is a measure one set $A_\alpha\in U_0$ such that for all $x\in A_\alpha$, there is $\xi<\kappa_n$, such that $(p_x,q_x)\Vdash_{\mathbb{L}_{\mu_x}} \dot{d}^{\mu_x}_{\sup x}\cap \alpha=\flev{\mu_x}{\alpha}{\xi}$.  Note that it follows that, for all $\alpha<\beta$ both in $S$, there are $\xi,\eta<\kappa_n$, $\tau<\kappa_0$ and $(p,q)\in \mathbb{L}_\tau$ such that $(p,q)\Vdash \flev{\tau}{\alpha}{\xi}=\flev{\tau}{\beta}{\eta}\cap\alpha$; this will be witnessed by any $x \in A$ with $\alpha,\beta \in x$.

Set $A:=\triangle_{\alpha\in S}A_\alpha$. Then $S, A$ are as desired.
\end{proof}
Fix $n, S,A, x\mapsto (p_x, q_x)$ as in the conclusion of the above lemma.

Much as in \cite{MagidorShelah} and later \cite{Neeman}, we require the notion of a \emph{system}.
\begin{definition}
Let $D \subseteq \Ord$, $\rho \in \Ord$, and $I$ be an index set.  A \defn{system on $D \times \rho$} is a family $\la R_s \ra_{s \in I}$ of transitive, reflexive relations on $D \times \rho$, so that
  \begin{enumerate}
    \item If $( \alpha, \xi ) R_s ( \beta, \zeta)$ and $( \alpha, \xi ) \neq ( \beta, \zeta)$ then $\alpha < \beta$.
    \item If $( \alpha_0, \xi_0)$ and $( \alpha_1, \xi_1)$ are both $R_s$-below $( \beta, \zeta)$, then $( \alpha_0,\xi_0)$ and $(\alpha_1,\xi_1)$ are comparable in $R_s$.
    \item For every $\alpha < \beta$ both in $D$, there are $s \in I$ and $\xi,\zeta \in \rho$ so that $( \alpha, \xi ) R_s ( \beta, \zeta )$.
  \end{enumerate}
A \defn{branch} through $R_s$ is a subset of $D \times \rho$ that is linearly ordered by $R_s$ and downwards $R_s$-closed (in particular, a branch is a partial function $b:D \rightharpoonup \rho$).
A \defn{system of branches through $\la R_s \ra_{s \in I}$} is a family $\la b_\eta \ra_{\eta \in J}$ so that each $b_\eta$ is a branch through some $R_{s(\eta)}$, and $D = \bigcup_{\eta \in J} \dom(b_\eta)$.
\end{definition}
As before, branches in a system need not be cofinal; however, note that now a branch $b_\eta$ through $R_s$ is cofinal iff $\dom(b_\eta)$ is cofinal in $D$.

Let $I=\{(\tau, p,q)\mid \tau<\kappa, \tau = \delta^+$ for some singular strong limit of countable cofinality $\delta$, and $(p,q)\in\mathbb{L}_\tau\}$.  Restricting to a final segment if necessary, we can assume for all $\alpha \in S$ that it is forced by $\1_{\mathbb{L}_\tau}$ that $\dot\sigma^\tau_\alpha(\xi) \neq \dot\sigma^\tau_\alpha(\eta)$ whenever $\xi \neq \eta$ are in $\kappa_n$.

For all $s=(\tau,p,q)\in I$, define the relation $R_s$ on $S\times\kappa_n$ by $(\alpha, \xi)R_s(\beta, \eta)$ iff $\alpha \leq \beta$ and $(p,q)\Vdash_{\mathbb{L}_\tau}\flev{\tau}{\alpha}{\xi}=\flev{\tau}{\beta}{\eta}\cap\alpha$.
\begin{prop}
$\la R_s \ra_{s \in I}$ is a system on $S\times \kappa_n$.

\end{prop}
\begin{proof}
That (1) and (2) hold is immediate by definition and the preceding paragraph, and the above lemma gives (3).
\end{proof}
\begin{lemma}\label{forcelemma2}
There exists, in $V[H]$, an unbounded $S' \subseteq S$ and a system of branches $\langle b_{s,\delta}\mid s\in I,\delta<\kappa_n\rangle$ through $\la R_s \rst S' \times \kappa_n \ra_{s \in I}$ such that each $b_{s,\delta}$ is a branch through $R_s \rst S' \times \kappa_n$.
\end{lemma}
\begin{proof}
Let $j=j_{U_{n+2}}:V\rightarrow M$.  Working in a forcing extension $V[H][H^*]$ of $V[H]$, where $H^*$ is generic for $\Col(\kappa_{n+1},j(\kappa_{n+3}))^V$, we may extend the embedding $j$ and regard it as a map $j:V[H]\rightarrow M^*$.  This poset is $\kappa_{n+1}$-closed in $V[\prod_{m \geq n+1} H_m]$, and $V[H]$ is a $\kappa_{n+1}$-c.c. extension of this model; in particular, $V[H]$ satisfies hypothesis (2) of the branch absorption Lemma 3.3 of \cite{Neeman}. Also, note that the poset to add the embedding is $<\kappa_{n+1}$ distributive in $V[H]$.

Let $\gamma\in j(S)\setminus\sup j"\mu$.  Since $\kappa_n<\crit(j)$, by elementarity applied to Lemma~\ref{forcelemma1}, we have for all $\alpha\in S$ that if we let $x \in j(A)$ so that $j(\alpha),\gamma \in x$, then there exist $\xi, \delta<\kappa_n$ and $s=(\mu_x, p_x,q_x)\in j(I)=I$ such that
 \[
 (p_x,q_x)\Vdash j(\flev{\mu_x}{\alpha}{\xi})=j\flev{\mu_x}{j(\alpha)}{\xi}=j\flev{\mu_x}{\gamma}{\delta}\cap j(\alpha).
 \]

For each $\delta<\kappa_n$ and $s=(\tau, p,q)\in I$, let
\[
 b_{s,\delta}=\{(\alpha,\xi)\mid \alpha\in S,\xi<\kappa_n, (p,q)\Vdash_{\mathbb{L}_\tau} j(\flev{\tau}{\alpha}{\xi})=j\flev{\tau}{\gamma}{\delta}\cap j(\alpha)\}.
 \]
We have that $\la b_{s, \delta}\mid s\in I, \delta<\kappa_n\ra$ is a system of branches through $\la R_s \ra_{s \in I}$: Each is clearly linearly ordered and downward closed; and we have just shown that any $x \in j(A)$ with $\gamma,j(\alpha) \in x$ witnesses $\alpha \in \dom{b_{s,\delta}}$ for some $\delta < \kappa_n$, so that $\bigcup \dom{b_{s,\delta}} = S$.  This system may not belong to $V[H]$, but we now satisfy precisely hypothesis (1) of the branch preservation lemma, Lemma 3.3, of \cite{Neeman}.  So there is some $(s,\delta) \in I \times \kappa_n$ so that $b_{s,\delta}$ is cofinal and belongs to $V[H]$.

Let $\D = \{(s, \delta) \mid b_{s,\delta} \in V[H]\}$.  By ${<}\kappa_{n+1}$-distributivity, $\D \in V[H]$, and we have just shown that $\D$ contains at least one pair $(s,\delta)$ corresponding to a cofinal branch. Again, by distributivity, $\langle b_{(s,\delta)}\mid (s,\delta)\in\D\rangle$ is in $V[H]$.  So the set $S' := \bigcup_{(s,\delta) \in \D} \dom(b_{s,\delta})$ is unbounded in $\mu$, and we have that $\la b_{s,\delta} \ra_{(s,\delta)\in\D}$ is a system of branches through $\la R_s \rst S' \times \kappa_n \ra_{s \in I}$.

Also, by passing to a subset of $I\times \kappa_n$ if necessary (any such will be in $V[H]$ by distributivity), we may assume that for all $s\in I$ and $\eta<\delta<\kappa_n$, if $b_{s,\eta}$ and $b_{s,\delta}$ are both cofinal, then they are distinct. This is done by simply removing duplicates (which may exit if the splitting between $j\dot{\sigma}^\tau_\gamma(\eta)$ and $j\dot{\sigma}^\tau_\gamma(\delta)$  is forced to be above $\sup j"\mu$).

\end{proof}

Note that if $s' = (\tau,p',q')$ and $s = (\tau,p,q)$ are such that $(p',q')\leq (p,q)$, then $R_{s'} \supseteq R_s$ and $b_{s',\delta} \supseteq b_{s,\delta}$ for all $\delta$; similarly, if $b_{s,\delta}$ is cofinal in $R_s$, then $(s,\delta) \in \D$ implies that $(s',\delta) \in \D$.

For each $s=(\tau,p,q)$ and $\delta$ such that $(s,\delta) \in \D$, let us fix an $\mathbb{L}_\tau$-name $\dot{\pi}_{s,\delta} = \bigcup\{ \flev{\tau}{\alpha}{\xi} \mid (\alpha,\xi) \in b_{s,\delta}\}$.  Note if $b_{s,\delta}$ is a cofinal branch, then $\dot{\pi}_{s,\delta}$ is forced by $(p,q)$ to name a cofinal branch through $\dot{d}^\tau$.

Our next goal is to show that these ground model branches are enough for us to repeat the final argument of Theorem~\ref{baby}.  What we need is a strengthened version of \ref{forcelemma1} for those branches from $\D$.

First, we bound the splitting for all branches (not just those in $V[H]$).  Working in $V[H][H^*]$ from the proof of Lemma~\ref{forcelemma2}, let, for each $\eta<\delta<\kappa_n$ and $s \in I$, let $\alpha_{s,\eta,\delta}$ be the least $\alpha$ so that $b_{s,\eta}(\alpha)$ and $b_{s,\delta}(\alpha)$ are (both defined and) not equal, if such exists; otherwise, let $\alpha_{s,\eta,\delta} = \sup \dom(b_{s,\eta}) \cup \dom(b_{s,\delta})$.  Let $\bar{\alpha} = \sup_{s \in I, \eta<\delta<\kappa_n} \alpha_{s,\eta,\delta} + 1$.

For $x \in A'$ let $(p_x,q_x) \in \mathbb{L}_{\mu_x}$ be as in the conclusion of Lemma~\ref{forcelemma1}; and set $s_x = (\mu_x,p_x,q_x)$.

\begin{lemma}\label{forcelemma3}
  There exist an unbounded $\bar{S} \subseteq S'$ and $\bar{A} \in U_0$ with $\bar{A} \subseteq A$, so that for all $x \in \bar{A}$, for all $\alpha \in \bar{S} \cap x$ we have
  \begin{equation*}\label{dagger} \tag{$\dagger_{x,\alpha}$}
  \text{for some } \delta < \kappa_n, \;(s_x,\delta) \in \D\text{ and }
    (p_x,q_x) \Vdash_{\mathbb{L}_{\mu_x}} \dot{d}^{\mu_x}_{\sup x} \cap \alpha = \dot{\pi}_{s_x,\delta} \cap \alpha.
   \end{equation*}
\end{lemma}
First let us see how to finish the proof of Theorem~\ref{forcebaby} assuming the lemma.  By our choice of $\bar{\alpha}$, any names $\dot{\pi}_{s,\delta}$ and $\dot{\pi}_{s,\eta}$ corresponding to cofinal branches of $V[H]$ are, by elementarity and the definition of these names, forced outright to disagree below $\bar{\alpha}$.

Suppose we have $x \in \bar{A}$ with $x\cap\bar{S}$ unbounded in $\sup x$, $\bar{\alpha} < \sup x$, and let $G_x$ be generic for $\mathbb{L}_{\mu_x}$ with $(p_x,q_x) \in G_x$.  By our definition of $\bar{\alpha}$, there exists a strengthening $(p'_x,q'_x) \in G_x$ of $(p'_x,q'_x)$ that forces $\dot{\pi}_{s,\delta} \cap \alpha \neq \dot{\pi}_{s,\eta} \cap \alpha$ for all $\alpha > \bar{\alpha}$ and distinct $\eta,\delta$ such that $(s,\eta),(s,\delta) \in \D$ represent cofinal branches; and $\alpha$ is above the domains of all bounded $b_{s,\delta}$'s.

Now for any $\alpha \in x \cap \bar{S}$, $\alpha>\bar{\alpha}$, we have some $\delta < \kappa_n$ so that $(p_x,q_x) \Vdash \dot{d}^{\mu_x}_{\sup x} \cap \alpha = \dot{\pi}_{s,\delta } \cap \alpha = \dot{\pi}_{s',\delta} \cap \alpha$.  Since we are above the splitting, these must be the same branch, and so without loss of generality this $\delta$ must be the same for each $\alpha \in x \cap \bar{S}$.  It follows that $(p_x,q_x) \Vdash \dot{d}^{\mu_x}_{\sup x} = \dot{\pi}_{s,\delta} \cap \sup x$.

Letting, for $(s,\delta) \in \D$ with $s = (\tau,p,q)$,
 \[
  T_{s,\delta} = \{\gamma \in \mu \mid (p,q) \Vdash_{\mathbb{L}_\tau} \dot{d}^{\mu_x}_\gamma = \dot{\pi}_{s,\delta} \cap \gamma \},
 \]
what we have shown is that $T := \bigcup_{(s,\delta) \in \D} T_{s,\delta} \supseteq \{ \sup x \mid x \in \bar{A}, \bar{\alpha}<\sup x, x\cap\bar{S}\text{ unbounded in }\sup x \}$.  So $T$ is stationary; since $|\D| \leq \kappa_n < \mu$, there is some fixed $(s,\delta)$ so that $T_{s,\delta}$ is stationary.  But since $\mathbb{L}_{\tau}$ preserves stationarity of subsets of $\mu$, we have that $b_{s,\delta}$ defines an ineffable branch through ${\dot{{d^\tau}}}$ in any extension containing $(p,q)$, a contradiction.

\begin{proof}[Proof of Lemma~\ref{forcelemma3}]
It is sufficient to show that if we set $A_\alpha = \{x \in A \mid \text{(\ref{dagger}) holds}\}$, then $\bar{S} = \{\alpha<\mu \mid A_\alpha \in U_0\}$ is unbounded in $\mu$; since in that case, $\bar{A} = A \cap \triangle_{\alpha \in \bar{S}} A_\alpha$ is as desired.

So suppose $\bar{S}$ is bounded.  Fix $\alpha_0 < \mu$ so that $\bar{\alpha} < \alpha_0$, and $A_\alpha \notin U_0$ whenever $\alpha > \alpha_0$ is in $S'$.  Put $A' := A \cap \triangle_{\alpha_0<\alpha \in S'} \mathcal{P}_{\kappa_0}(\mu) \setminus A_\alpha$; so $A' \in U_0$, and (\ref{dagger}) fails whenever $\alpha \in x \in A'$ with $\alpha \in S'$.

%, there is an $\tilde{\alpha} \in S' \cap x$ with $\tilde{\alpha} > \alpha_0$ such that whenever $\delta$ is such that $(s,\delta) = ((\mu_x,p_x,q_x), \delta) \in \D$, we have
%  \[
%  (p_x,q_x)\not\Vdash \dot{d}^{\mu_x}_{\sup x} \cap \tilde{\alpha} = \dot{\pi}_{s,\delta} \cap \tilde{\alpha}.
%  \]
%Let $A' \in U_0$ be the set of such $x$.
%
% Recalling the definition of the branches $\dot{\pi}_{s,\delta}$, let $j:V[H] \to M^*$ be the embedding with critical point $\kappa_{n+2}$ added by $\Col(\kappa_{n+1},<j(\kappa_{n+2}))$ as in the proof of Lemma~\ref{forcelemma2}.  Working in this extension, we can bound all possible instances of splitting (as above, but this time including $(s, \delta) \notin \D$), so let us assume that if $s=(\tau,p,q)$ and $(p,q) \Vdash \dot{\pi}_{s,\delta} \cap \alpha = \dot{\pi}_{s,\eta} \cap \alpha$ for some $\alpha$, then the least such $\alpha$ is below $\alpha_0$.

We wish to show that if $R'_{s}$ is obtained by deleting all ground model branches $b_{s,\delta}$ from $R_s$, then the resulting family $\la R'_{s} \ra_{s \in I}$ is a system on $(S' \setminus \alpha_0) \times \kappa_n$.  That is, for each $s$, let
  \begin{align*}
  (\alpha,\xi) \mathbin{R'_s} (\beta,\zeta) \iff &\alpha_0 < \alpha,\; \beta \in S', \;(\alpha,\xi) \mathbin{R_s} (\beta,\zeta), \text{ and for all }\delta < \kappa_n,\\
  &\text{ if }(s,\delta) \in \D \text{ then }(\alpha,\xi) \notin b_{s,\delta}.
  \end{align*}
The first two properties of a system are clear.  For (3), suppose $\alpha_0 < \alpha < \beta$ with $\alpha,\beta$ both in $S'$.  By elementarity, we have some $x' \in j(A')$ so that $j(\alpha),j(\beta),\gamma \in x'$, where here $\gamma$ is the element of $j(S)\setminus \sup j"\mu$ we used to define the system of branches in Lemma~\ref{forcelemma2}.  Now by Lemma \ref{forcelemma1}, and since $j(\alpha),j(\beta)\in j(S)$ as well, we have some $\xi,\zeta,\delta < \kappa_n$ so that
 \[
 (p_{x'},q_{x'}) \Vdash j\flev{\mu_{x'}}{j(\alpha)}{\xi} = j\flev{\mu_x}{\gamma}{\delta} \cap j(\alpha), j\flev{\mu_x}{j(\beta)}{\zeta} = j\flev{\mu_x}{\gamma}{\delta} \cap j(\beta),
 \]
and moreover, each of these nodes is forced by $(p_{x'},q_{x'})$ to cohere with $j\dot{d}^{\mu_x}_{\sup x}$.

Now $x' \notin j" V[H]$, but by elementarity we have some $x \in A'$ so that $s_x = s_{x'}$ with $\alpha,\beta \in x$.  In particular, we have $(\alpha,\xi) R_{s_x,\delta} (\beta,\zeta)$; indeed, $(\alpha,\xi),(\beta,\eta)$ are both in $b_{s_x,\delta}$.  Since we are above splitting, we have for any $\delta'$ with $(\alpha,\xi) \in b_{s_x,\delta'}$ that this branch coincides with $b_{s_x,\delta}$.  So to obtain condition (3), we just need to show $(s_x,\delta) \notin \D$.  %Since all of these ordinals are above $\alpha_0$, we may take this $\delta$ to be unique.

Suppose $(s_x,\delta) \in \D$.  Since $\alpha \in x\cap S'$ with $x \in A'$, (\ref{dagger}) fails.  Then we have
  \[
   (p_x,q_x) \not\Vdash \dot{d}^{\mu_x}_{\sup x} \cap \alpha = \dot{\pi}_{s_x,\delta} \cap \alpha.
  \]
But by our choice of $x$,
  \[
   (p_x,q_x) \Vdash \dot{d}^{\mu_x}_{\sup x} \cap \alpha = \flev{\mu_x}{\alpha}{\xi}.
  \]
But these conditions and our definition of $\dot{\pi}_{s_x,\delta}$ contradict $(\alpha,\xi) \in b_{s_x,\delta}$.

%By our assumption for a contradiction, we have some $\tilde{\alpha} \in x \cap S'$ so that $\tilde{\alpha} > \alpha_0$.  By the same argument, we have $(\tilde{\alpha},\eta) \in b_{s_x,\delta}$ for some $\eta < \kappa_n$; that is, $(p_x,q_x) \Vdash \dot{d}^{\mu_x}_{\sup x} \cap \tilde{\alpha} = \dot{\pi}_{s_x,\delta} \cap \tilde{\alpha}$.  Then we have $(s_x,\delta) \notin \D$ by assumption.
%
%Clearly $(\alpha,\xi) R_{s_x} (\beta,\zeta)$, and since we are above the splitting, we have $(\alpha,\xi) \notin b_{s_x,\delta}$ for any $(s,\delta) \in \D$ (since there is at most one such $\delta$).  This shows $(\alpha,\xi) R'_{s_x} (\beta,\zeta)$ as needed.

Now that $\la R'_s \ra_{s \in I}$ is a system, we let for each $(s,\delta) \notin \D$ the branch $b'_{s,\delta}$ be the restriction of $b_{s,\delta}$ to $R'_s$.  Then $\langle b'_{s,\delta}\mid (s,\delta)\notin \D\rangle$ is a system of branches through the system $\langle R_s'\mid s\in I\rangle$.

Now recapitulating the argument of Lemma~\ref{forcelemma2}, we see that there exists some $s$ and a $\delta$ so that $b'_{s,\delta} \rst R'_s$ is cofinal and belongs to $V[H]$.  Since $b_{s,\delta}$ can be recovered from any cofinal subset, we must have $(s,\delta) \in \D$.  But this contradicts our definition of the system $\la R'_s \ra_{s \in I}$.  This contradiction completes the proof.
\end{proof}

Next we prove the full, two cardinal version.
We make a slight modification: For $\tau<\kappa=\kappa_0$, with countable cofinality, define $\mathbb{L}_\tau=\Col(\omega, \tau)\times \Col(\tau^{+3}, {<}\kappa)$.

\begin{theorem}\label{force two cardinal}
Let $H=\prod_{n}H_n$ be $\prod_{n<\om} \Col(\kappa_n, {<}\kappa_{n+1})$-generic over $V$.  Then there exists $\tau<\kappa$ with countable cofinality such that in the extension of $V[H]$ by $\mathbb{L}_\tau$, $\mu = \aleph_{\omega+1}$ and $\ITP_\mu$ holds.
\end{theorem}

\begin{proof}
Suppose otherwise. Then for every $\tau<\kappa$, there is some $\lambda$, such that $\ITP(\mu,\lambda)$ fails in the extension of $V[H]$ by $\mathbb{L}_\tau$. By taking a supremum, assume the $\lambda$ is the same for all $\tau$ and that $\lambda^{\mu}=\lambda$.

So for each $\tau$, let $\dot{d}^\tau$ be a name for a $\mathbb{P}_\mu(\lambda)$ thin list which is forced by $1_{\mathbb{L}_\tau}$ not to have an ineffable branch. Let $K$ be $\Col^{V[H]}(\mu, {<}\lambda)$-generic over $V[H]$.

\begin{theorem}\label{muplus}
There is $\tau<\kappa$, such that $V[H][K][\mathbb{L}_\tau]$, $d^\tau$ has an ineffable branch.
\end{theorem}

Assuming this, let us show that this is enough to prove Theorem \ref{force two cardinal}. Let $b$ be an ineffable branch for $d^\tau$ in $V[H][K][\mathbb{L}_\tau]$. Since $\Col(\mu, {<}\lambda)$ is $\mu$ closed in $V[H]$, and $\mathbb{L}_\tau$ is $\kappa$-c.c. (actually we only need $\mu$-c.c.), we have that in $V[H]$, $1_{\mathbb{L}_\tau}$ forces that $\Col(\mu, {<}\lambda)$ has the $\mu$-thin approximation property. Therefore, $b\in V[H][\mathbb{L}_\tau]$. Since stationarity is downwards absolute, $b$ is an ineffable branch for  $d^\tau$ in $V[H][\mathbb{L}_\tau]$. But that is a contradiction with the choice of $\dot{d}^\tau$, and so the result follows.
\end{proof}

Now for the proof of Theorem \ref{muplus}, first note that in $V[H][K]$, $\kappa$ is still supercompact, each $\kappa_n$, $n>0$, is generically supercompact for the right type of quotient, and $\lambda$ is $\mu^+$. So, this will be a  similar argument as in Theorem \ref{forcebaby}. Except that here we work in $V[H][K]$ instead of $V[H]$, and consider $\mathcal{P}_{\mu}(\lambda)=\mathcal{P}_\mu(\mu^+)$ lists and $\lambda=\mu^+$-supercompact embedding.

We outline the proof, skipping some of the details.

\begin{proof}[Proof of Theorem \ref{muplus}]
Let $U_0$ be a normal measure on $\mathcal{P}_\kappa(\lambda)$ in $V[H][K]$, and
set $i=j_{U_0}:V[H][K]\rightarrow M$ to be the corresponding $\lambda$ supercompact embedding with critical point $\kappa$. And for each $n>0$, let $U_n$ be a normal measure on $\mathcal{P}_{\kappa_n}(\lambda)$ in $V$.

For each $x\in\mathcal{P}_\mu(\lambda)$, set $\tau_x=\kappa_x^{+\omega}$. Then $[x\mapsto \mathbb{L}_{\tau_x}]_{U_0}=\Col(\omega, \nu)\times \Col(\mu^{++}, {<}i(\kappa))=\mathbb{L}_\nu$ (of course here $\mu^{++}=\kappa^{+\omega+3}=\lambda^+$ both in $V[H][K]$ and in the ultrapower.

Note that $\mathcal{P}^{V[H]}_\mu(\lambda)=\mathcal{P}^{V[H][K]}_\mu(\lambda)$, so we just denote it by $\mathcal{P}_\mu(\lambda)$.
As before for each $z\in \mathcal{P}_\mu(\lambda)$, assume that that $1_{\mathbb{L}_{\tau_x}}$ forces that the $z$-th level of $\dot{d}^{\tau_x}$ is enumerated by the names $\{\flev{\tau_x}{z}{\xi}\mid \xi<\nu\}$.

As in Theorem \ref{two cardinal}, fix a bijection $c:\mathcal{P}_\mu(\lambda)\rightarrow\lambda$, and set $z^*:=\bigcup i"\mathcal{P}_\mu(\lambda)=[g]_{U_0}$. Then we have:
\begin{itemize}
\item
If $A\in U_0$, then $\bar{A}:=\{g(x)\mid x\in A\}$ is stationary in $\mathcal{P}_\mu(\lambda)$.
\item
$i\dot{d}^\nu_{z^*}=[x\mapsto \dot{d}^{\tau_x}_{g(x)}]_{U_0}$.
\item
We may assume that for all $x$, $g(x)=\bigcup\{z\mid c(z)\in x\}$.

\end{itemize}

\begin{lemma}\label{lemma1muplus}
There exist $n<\omega$, a cofinal $S\subset\mathcal{P}_\mu(\lambda)$, and $A\in U_0$ such that for all $x\in A$ and $z\in S$ with $c(z)\in x$, there is $\xi<\kappa_n$ and $(p_x,q_x)\in \mathbb{L}_{\tau_x}$ such that $(p_x,q_x)\Vdash \dot{d}^{\tau_x}_{g(x)}\cap z=\flev{\tau_x}{z}{\xi}$.
\end{lemma}
\begin{proof}
Analogously to the proof of Lemma \ref{forcelemma1}. Here for every $z\in\mathcal{P}_\mu(\lambda)$, we find conditions $(p^z, q^z)$ forcing over $\mathbb{L}_\nu$ that $i\dot{d}^\nu_{z^*}\cap i(z)=i\flev{\nu}{i(z)}{\xi}$ for some $\xi<i(\nu)$. Since the closure of the second factor is $\lambda^+$, we can define the $q^z$'s inductively to be decreasing according to some enumeration and take $q\in \Col(\mu^{++}, {<}i(\kappa))$ to be a lower bound. And find a cofinal $S\subset\mathcal{P}_\mu(\lambda)$, $n<\omega$ and $p\in\Col(\omega, \nu)$, such that for all $z\in S$, the witnessing $\xi$ is less than $i(\kappa_n)$ and $p^z=p$.

Let $p=[x\mapsto p_x]$, $q=[x\mapsto q_x]$, and for each $z\in S$, get a measure one set $A_z$ witnessing $(p, q)\Vdash_{\mathbb{L}_\nu}i\dot{d}^\nu_{z^*}\cap i(z)=i\flev{\nu}{i(z)}{\xi}$ for some $\xi<i(\kappa_n)$.
Take $\triangle A_z$. These are as desired.
\end{proof}

Let $I=\{(\tau, p,q)\mid \tau<\kappa,(p,q)\in\mathbb{L}_\tau\}$.
For all $s=(\tau,p,q)\in I$, define the relation $R_s$ on $S\times\kappa_n$ by $(z, \xi)R_s(z', \eta)$ iff $z \subset z'$ and $(p,q)\Vdash_{\mathbb{L}_\tau}\flev{\tau}{z}{\xi}=\flev{\tau}{z'}{\eta}\cap z$.
Then,
$\la R_s \ra_{s \in I}$ is a system on $S\times \kappa_n$. Here the definition of system uses $\subset$ instead of $\leq$. More precisely:

\begin{definition}
Let $D$ be a cofinal subset of $\mathcal{P}_\mu(\lambda)$, $\rho \in \Ord$, and $I$ be an index set.  A \it{system on $D \times \rho$} is a family $\langle R_s \rangle_{s \in I}$ of $\subset$-transitive, reflexive relations on $D \times \rho$, so that
  \begin{enumerate}
    \item If $( x, \xi ) R_s ( y, \zeta)$ and $( x, \xi ) \neq ( y, \zeta)$ then $x\subsetneq y$.
    \item If $( x_0, \xi_0)$ and $( x_1, \xi_1)$ are both $R_s$-below $( y, \zeta)$ and $x_0\subset x_1$, then $( x_0,\xi_0) R_s(x_1,\xi_1)$.
    \item For every $x,y$ both in $D$, there are $z\in D$, $s \in I$ and $\xi,\xi',\zeta \in \rho$ so that $x\cup y\subset z$, $(x, \xi ) R_s (z, \zeta )$ and $(y,\xi')R_s(z,\zeta)$.
  \end{enumerate}
A \it{branch} through $R_s$ is a partial function $b:D \rightharpoonup \rho$, such that
\begin{enumerate}
\item
if $x\subset y$ are both in $\dom(b)$, then $(x, b(x))R_s(y, b(y))$.
\item
if $y\in \dom(b)$, and $(x,\xi) R_s (y,b(y))$, then $(x,\xi)\in b$, i.e. $b$ is downwards $R_s$-closed
\end{enumerate}

A \it{system of branches} through $\langle R_s \rangle_{s \in I}$ is a family $\langle b_\eta \rangle_{\eta \in J}$ so that each $b_\eta$ is a branch through some $R_{s(\eta)}$, and $D = \bigcup_{\eta \in J} \dom(b_\eta)$.
\end{definition}

Next we show the branch preservation lemma we will use. The proof follows closely Lemma 3.3 of \cite{Neeman}, adapted to the two cardinal version. We include it for completeness. The key fact is that if a forcing adds a branch, then this branch is thinly $\mu$-approximated.

\begin{lemma}\label{branch}
Let $V\subset W$ be models of set theory and  $W$ is a $\tau$-c.c. forcing extension of $V$ and $\mathbb{Q}\in V$ is $\tau$-closed in $V$.
In $W$ suppose $\langle R_s \rangle_{s \in I}$ is a system on $D\times \rho$, for some cofinal $D\subset\mathcal{P}_\mu(\lambda)$, such that forcing with $\mathbb{Q}$ over $W$ adds a system of branches $\langle b_j \rangle_{j\in J}$ through this system. Finally suppose $\chi:=\max(|J|, |I|, \rho)^+<\tau<\mu$. Then there is a cofinal branch $b_j\in W$.

\end{lemma}

\begin{proof}
Suppose otherwise; say $1_{\mathbb{Q}}$ forces that each $\dot{b}_j\notin W$, where $\dot{b}_j$ is a $\mathbb{Q}$-name in $W$ for the $j$-th branch.
Let $G=\langle G_\xi\mid \xi<\chi\rangle$ be $\mathbb{Q}^{\chi}$-generic over $W$. Here  $\mathbb{Q}^{\chi}$ is the full support $\chi$ power of $\mathbb{Q}$. For every $\xi<\chi, j\in J$, let $b^\xi_j=\dot{b}_j[G_\xi]$.

First note that since $\mathbb{Q}^\chi$ is $\tau$-closed in $V$, it must be $<\tau$ distributive in $W$. Working in $W[G]$, for each non cofinal branch $b^\xi_j$, there is $z^\xi_j\in\mathcal{P}^W_\mu(\lambda)$ such that for all $z\supset z^\xi_j$, $z\notin\dom(b^\xi_j)$. Let $z_0$ be their union. By distributivity, $z_0\in \mathcal{P}^W_\mu(\lambda)$. Similarly, we can find $z_1\supset z_0$, in $\mathcal{P}^W_\mu(\lambda)$, such that  for all cofinal $b_j^\xi, b_j^\eta$
for $\xi<\eta$, for all $z\supset z_1$, $b_j^\xi(z)\neq b_j^\eta(z)$ (possibly because one of them is not defined) if they are branches through the same relation. This splitting follows by mutual genericity.

Let $z\in \mathcal{P}^W_\lambda(\mu)$, $z_1\subset z$ be such that $z\in D$. Since we have a system of branches, for every $\xi<\chi$, there is $j_\xi\in J$, such that $z\in \dom(b^\xi_{j_\xi})$ and $b^\xi_{j_\xi}$ is a branch through $R_{s_\xi}$. Let $\alpha_\xi=b^\xi_{j_\xi}(z)$. Then the map $\xi\mapsto (j_\xi, s_\xi, \alpha_\xi)$ is from $\chi\rightarrow |J|\times |I|\times\rho$, and $|J|,|I|,\rho<\chi$. So there are $\xi<\eta<\chi$ and $j,s$ such that $j=j_\xi=j_\eta$ and $s=s_\xi=s_\eta$. But then $b^\xi_j(z)=b^\eta_j(z)$, and they are both branches through $R_s$. Contradiction.
\end{proof}

\begin{lemma}\label{lemma2muplus}
There exists, in $V[H][K]$, a cofinal $S' \subseteq S$ and a system of branches $\langle b_{s,\delta}\mid s\in I,\delta<\kappa_n\rangle$ through $\la R_s \rst S' \times \kappa_n \ra_{s \in I}$ such that each $b_{s,\delta}$ is a branch through $R_s \rst S' \times \kappa_n$.
\end{lemma}
\begin{proof}
Let $j=j_{U_{n+2}}:V\rightarrow M$ be the $\lambda$-s.c. embedding obtained from $U_{n+2}$.  We lift $j$ to $j:V[H][K]\rightarrow M^*$  in a forcing extension $V[H][K][H^*]$ of $V[H][K]$, where $H^*$ is generic for $\Col(\kappa_{n+1},j(\kappa_{n+3}))^V$.  This poset is $<\kappa_{n+1}$-distributive in $V[H][K]$, $\kappa_{n+1}$-closed in $V[\prod_{m \geq n+1} H_m][K]$, and $V[H][K]$ is a $\kappa_{n+1}$-c.c. extension of this model. In particular, $V[H][K]$ satisfies hypothesis of the branch lemma above.

Let $u\in j(S)$ be such that $u\supset j"\mathcal{P}_\mu(\lambda)$.
For each $\delta<\kappa_n$ and $s=(\tau, p,q)\in I$, let
\[
 b_{s,\delta}=\{(z,\xi)\mid z\in S,\xi<\kappa_n, (p,q)\Vdash_{\mathbb{L}_\tau} j(\flev{\tau}{z}{\xi})=j\flev{\tau}{u}{\delta}\cap j(z)\}.
 \]
$\la b_{s, \delta}\mid s\in I, \delta<\kappa_n\ra$ is a system of branches through $\la R_s \ra_{s \in I}$. By passing to a subset of $I\times\kappa_n$ if necessary, we may assume that any two $b_{s,\eta}$ and $b_{s,\delta}$ are distinct.

Set $\D = \{(s, \delta) \mid b_{s,\delta} \in V[H][K]\}\in V[H][K]$ by $<\kappa_{n+1}$-distributivity, as is $\langle b_{(s,\delta)}\mid (s,\delta)\in\D\rangle$. By Lemma \ref{branch}, there is a confinal branch among these in $V[H][K]$.  So the set $S' := \bigcup_{(s,\delta) \in \D} \dom(b_{s,\delta})$ is cofinal in $\mathcal{P}_\kappa(\lambda)$, and we have that $\la b_{s,\delta} \ra_{(s,\delta)\in\D}$ is a system of branches through $\la R_s \rst S' \times \kappa_n \ra_{s \in I}$.
\end{proof}

Let $\dot{\pi}_{s,\delta}$ be the corresponding name for a branch obtained from $b_{s,\delta}$.
Fix $\bar{z}\in\mathcal{P}_\mu(\lambda)$, such that for any two distinct $(s,\eta)$ and $(s,\delta)$ in $\mathcal{D}$,
the names $\dot{\pi}_{s,\delta}$ and $\dot{\pi}_{s,\eta}$ are forced densely often  to disagree at $\bar{z}$, if they are cofinal; and are bounded by $\bar{z}$ otherwise.
\begin{lemma}\label{lemma3 muplus}
  There exist an unbounded $\bar{S} \subseteq S'$ and $\bar{A} \in U_0$ with $\bar{A} \subseteq A$, so that for all $x \in \bar{A}$, for all $z \in \bar{S}, c(z)\in x$, there exists a $\delta < \kappa_n$ so that, letting $s_x = (\tau_x,p_x,q_x)$, we have $(s_x,\delta) \in \D$, and
   \[
   (p_x,q_x) \Vdash_{\mathbb{L}_{\tau_x}} \dot{d}^{\tau_x}_{g(x)} \cap z = \dot{\pi}_{s_x,\delta} \cap z.
   \]
\end{lemma}
\begin{proof}
The proof is as in Lemma \ref{forcelemma3}. Again, we assume otherwise. Let $(\dagger_{x,z})$ be the statement that
there exists a $\delta < \kappa_n$ so that, letting $s_x = (\tau_x,p_x,q_x)$, we have $(s_x,\delta) \in \D$, and
   \[
    (p_x,q_x) \Vdash_{\mathbb{L}_{\tau_x}} \dot{d}^{\tau_x}_{g(x)} \cap z = \dot{\pi}_{s_x,\delta} \cap z.
   \]
Then there must be some $z_0 < \mathcal{P}_\mu(\lambda)$,  $\bar{z} \subset z_0$, a measure one $A'\in U_0$, such that $(\dagger_{x,z})$ fails whenever $c(z) \in x \in A'$ with $z \in S'$. Define a system $\la R'_{s} \ra_{s \in I}$  on $\{z\in S'\mid z_0\subset z\} \times \kappa_n$ from $\la R_{s} \ra_{s \in I}$  by deleting all ground model branches $b_{s,\delta}$ as follows:

  \begin{align*}
  (z,\xi) \mathbin{R'_s} (z',\zeta) \iff &z_0 \subset z,\; z' \in S', \;(z,\xi) \mathbin{R_s} (z',\zeta), \text{ and for all }\delta < \kappa_n,\\
  &\text{ if }(s,\delta) \in \D \text{ then }(z,\xi) \notin b_{s,\delta}.
  \end{align*}

Setting $b'_{s,\delta}$ to be the restriction of $b_{s,\delta}$ to $R'_s$, we get that $\langle b'_{s,\delta}\mid (s,\delta)\notin \D\rangle$ is a system of branches through  $\la R'_s \ra_{s \in I}$.

But then, by branch preservation, there exists some $(s,\delta)$ so that $b'_{s,\delta} \rst R'_s$ is cofinal and belongs to $V[H][K]$, and so $(s,\delta) \in \D$.  Contradiction with the definition of $\la R'_s \ra_{s \in I}$.
\end{proof}

Finally, we can complete the argument.

For all $x \in \bar{A}$ with $c(\bar{z})\in x$, there is a unique $\delta<\kappa_n$, such that for all $z\in \bar{S}$, with $c(z)\in x$, $(p_x,q_x) \Vdash \dot{d}^{\tau_x}_{g(x)} \cap z = \dot{\pi}_{s,\delta } \cap z$.
Let $A''=\{x\in \bar{A}\mid c(\bar{z})\in x, g(x)=\bigcup\{z\mid \bar{z}\subset z, c(z)\in x\}\}\in U_0$. Then for all $x\in A''$, there is some $\delta<\kappa_n$, such that $(p_x,q_x) \Vdash \dot{d}^{\tau_x}_{g(x)}  = \dot{\pi}_{s,\delta } \cap g(x)$.

For $(s,\delta) \in \D$ with $s = (\tau,p,q)$,
 \[
  T_{s,\delta} := \{z \in \mathcal{P}_\mu(\lambda) \mid (p,q) \Vdash \dot{d}^{\tau}_z = \dot{\pi}_{s,\delta} \cap z \},
 \]
We have shown is that $T := \bigcup_{(s,\delta) \in \D} T_{s,\delta} \supseteq \{ g(x) \mid x\in A''\}$ is stationary. Since $|\D| \leq \kappa_n < \mu$, for some $(s,\delta)$, $T_{s,\delta}$ is stationary.  Denoting $s=(\tau,p,q)$, it follows that $b_{s,\delta}$ generates an ineffable branch through ${\dot{{d^\tau}}}$ in any extension containing $(p,q)$, a contradiction.
\end{proof}

\section{Open problems}

Having obtained ITP at the successor of a singular, the direction of forcing ITP at successive cardinals past a singular looks very promising. As a first step one can try and combine the construction in \cite{Neeman} with the results in this paper to force ITP at every $\aleph_n$, $n>1$ together with ITP at $\aleph_{\omega+1}$. We conjecture that this should be possible.

Next, as mentioned in the introduction, in order to get the tree property everywhere one needs failures of SCH. The reason is that the tree property (or strong tree property or ITP) at $\kappa^{++}$ for a singular strong limit $\kappa$ implies that SCH fails at $\kappa$. So here are some questions to consider:
\begin{question}
Can we obtain $\ITP$ at $\kappa^+$ for a singular strong limit cardinal $\kappa$ together with failure of SCH at $\kappa$?
\end{question}

\begin{question}
Can we obtain $\ITP$ at $\kappa^+$ and $\kappa^{++}$ for a singular strong limit cardinal $\kappa$?
\end{question}

\begin{question}
Can we get the above for $\kappa=\aleph_{\omega^2}$? Or much more ambitiously, for $\kappa=\aleph_\omega$?
\end{question}

\bibliographystyle{asl}

\end{document}